\documentclass[12pt]{amsart}

\usepackage{amsmath,amsthm,amsfonts,amssymb,epsfig}
\usepackage[all]{xy}

\newcommand{\mbb}[1]{\mathbb{#1}}

\newcommand{\mc}[1]{\mathcal{#1}}

\newcommand{\hra}{\hookrightarrow}

\def\<{\left<}
\def\>{\right>}

\theoremstyle{plain}

\newtheorem{thm}{Theorem}[section]
\newtheorem{defn}[thm]{Definition}
\newtheorem{lem}[thm]{Lemma}
\newtheorem{cor}[thm]{Corollary}
\newtheorem{prop}[thm]{Proposition}

\newtheorem{example}[thm]{Example}

\newtheorem*{example*}{Example}
\newtheorem*{thm*}{Theorem}
\newtheorem*{rem*}{Remark}
\newtheorem*{lem*}{Lemma}

\newtheorem*{cor*}{Corollary}
\newtheorem*{prop*}{Proposition}

\theoremstyle{remark}

\newtheorem{rem}[thm]{Remark}

\newcommand{\ov}[1]{\overline{#1}}

\newcommand{\mo}[1]{\operatorname{#1}}
\newcommand{\up}[1]{\!\: ^{#1}\!\,}

\newcommand{\bP}{\mbb P}

\title{Corestrictions of algebras and splitting fields}

\author{Daniel Krashen}

\begin{document}

\begin{abstract}
Given a field $F$, an \'etale extension $L/F$ and an Azumaya algebra
$A/L$, one knows that there are extensions $E/F$ such that $A \otimes_F
E$ is a split algebra over $L \otimes_F E$. In this paper we bound the
degree of a minimal splitting field of this type from above and show
that our bound is sharp in certain situations, even in the case where
$L/F$ is a split extension. This gives in particular a number of
generalizations of the classical fact that when the tensor product of
two quaternion algebras is not a division algebra, the two quaternion
algebras must share a common quadratic splitting field.

In another direction, our constructions combined with results in
\cite{Kar:C2ind} also show that for any odd prime number $p$, the generic
algebra of index $p^n$, and exponent $p$ cannot be expressed
nontrivially as the corestriction of an algebra over any extension
field if $n < p^2$.
\end{abstract}

\maketitle

\section{Introduction}

It is a classical fact due to Albert that two quaternion algebras
over a field whose tensor product has index at most two must share a
common quadratic splitting field. In this paper we give generalizations
to this fact in two different directions. On the one hand, we obtain
certain generalizations of this statement for algebras of higher degree
(see corollary \ref{main_cor} and example \ref{split_example}), which
are philosophically similar to, but not intersecting with the results in
\cite{Kar:CSF}. On the other hand, we consider the idea that a pair of
algebras may be regarded as an Azumaya algebra over a split \'etale
extension of the form $F \times F$. This leads to analogous results in
the case of Azumaya algebras over more general \'etale extensions (see
theorems \ref{general_thm}, \ref{main_thm} and examples
\ref{nonsplit_quat_example}, \ref{nonsplit_example}). For example, we
derive the following analog of the above fact: given a quadratic field
extension $L/F$ and a quaternion division algebra $Q$ over $L$ whose
corestriction to $F$ has index at most $2$, then $Q$ has a splitting
field of the form $E \otimes_F L$ for some quadratic field extension
$E/L$.

Recall that a finite dimensional algebra $A$ over a field $F$ is called
separable if it is an Azumaya algebra over its center, which is a finite
dimensional \'etale extension of $F$. Setting $L = Z(A)$, we may write
$L = L_1 \times \cdots \times L_m$ as a product of separable field
extensions of $F$, and hence an \'etale algebra over $F$. Note that such
an algebra $A$ may itself be written as a product $A = A_1 \times \cdots
\times A_m$ where $A_i$ is a central simple $L_i$ algebra.

One knows very well the minimal degrees of \'etale extensions $K/L$ such
that $A \otimes_L K$ is split: such an extension may be written as $K =
K_1 \times \cdots K_m$ with $K_i$ an extension of $L_i$ and one may
always find $K$ with $[K_i : L_i] = \mo{ind} A_i$. On the other hand,
although one knows in principle that there are finite extensions $E/F$
such that $A \otimes_F E$ is a split algebra over $L \otimes_F E$, (one
may easily see this by considering extension of scalars to an algebraic
closure of $F$), it may be quite difficult to compute the minimal degree
of such an extension. In the case that $L/F$ is a split \'etale
extension, we may write $A = A_1 \times \cdots \times A_m$ and it follows
that such an $L$ is exactly a common splitting field for each of the
algebras $A_m$. Even in this split case determining a minimal degree for
$L/F$ with this property is quite delicate, and an explicit answer is
not known in general (however, see \cite{Kar:CSF} for various results in
this direction).

In this paper we give a construction of \'etale splitting fields for
separable algebras (theorem \ref{general_thm}), which we show can
provide in some sense optimal bounds on the degrees of splitting
extensions (theorem \ref{main_thm}, \ref{main_cor}, and propositions
\ref{prop_1}, \ref{prop_2}). These results generalize the classical fact
that two quaternion algebras whose product has index at most two share a
common quadratic subfield. Our main results are as follows:

\begin{thm*}[\ref{general_thm}]
Let $L/F$ an \'etale extension of dimension $m$ and and $A$ an Azumaya
algebra over $L$ of degree $(d_1, \ldots, d_m)$ (see definition
\ref{deg_def}). Let $I$ be the index of $\mo{cor}_{L/F} A$ and let $P$
be its period. Let $r$ be the remainder upon dividing $\sum d_i - m$ by
$I$. Then there exists an \'etale extension $E/F$ of degree \[\frac{(d_1
+ \cdots + d_m - m)!}{(d_1 - 1)! \cdots (d_m - 1)!} \, P^r\] such that
$A \otimes_F E$ is split as an $L \otimes_F E$ algebra.
\end{thm*}

In the case of prime power degrees, one may say something about the form
of this degree:

\begin{thm*}[\ref{main_thm}]
Let $p$ be a prime number, $L/F$ an \'etale extension of degree $p^k$
and $A$ an Azumaya algebra over $L$ of constant degree $p^n$ such that
$\mo{cor}_{L/F} A$ has index dividing $p^k$. Then there exists an \'etale
extension $E/F$ of degree $p^{n(p^k - 1)}m$ where $m$ is relatively
prime to $p$ such that $A \otimes_F E$ is split as an $L \otimes_F E$
algebra.
\end{thm*}

One may always show that the algebra $A$ may be split by such an \'etale
extension $E/F$ with $[E:F] = p^{n p^k}$ (see proposition
\ref{a_priori_split}).  The content of the theorem \ref{main_thm} is
that after making assumptions about the index, we may in fact do better.

In the case $L/F$ is a split \'etale extension, we obtain the following
as a corollary (via lemma \ref{cor_product}):
\begin{cor*}[\ref{main_cor}]
Let $p$ be a prime number, and let $A_1, \ldots, A_{p^k}$ be central
simple algebras of degree $p^n$ over a field $F$ such that $A_1 \otimes
\cdots \otimes A_p$ has index dividing $p^k$. Then there exists an
\'etale extension $E/F$ of degree $p^{n(p^k-1)}m$ where $m$ is
relatively prime to $p$ such that $E$ is a splits each of the algebras
$A_1, \ldots, A_p$.
\end{cor*}

\begin{rem*} \label{what_is_m}
It follows from the proof of the theorem that the integer $m$ appearing
in the statements above may be explicitly expressed as:
\[m = \frac{(p^k(p^n - 1))!}{((p^n - 1)!)^{p^k} p^{n(p^k - 1)}}\]
\end{rem*}

The classical case of quaternion algebras corresponds to $p = 2, n = 1,
k = 1$ with a split quadratic \'etale extension. To get a feeling for this
result, we provide a few examples for small values:

\begin{example}[$k = n = 1$, $p = 2$] \label{nonsplit_quat_example}
Suppose $L/F$ is a separable quadratic field extension, and $Q$ is a
quaternion algebra over $L$ such that $\mo{cor}_{L/F} Q$ is not
division. Then there exists a quadratic field extension $E/F$ such that
$Q \otimes_F E$ is split over $L \otimes_F E$.
\end{example}
This particular example is already known and may be proved using
quadratic form theory, as was pointed out to the author by A. Merkurjev.
In particular, one may see this by considering the corestriction of the
pure part of the $2$-fold Pfister form associated to $Q$ with respect to
a linear map $L \to F$ taking $1$ to $0$. Since the result must be isotropic
by assumption on the index, the original form must represent an element
of the ground field. This implies in the symbol $(a, b)$ defining $Q$,
we may take one of the elements, say $a$ to lie in the ground field,
providing us with the splitting field $F(\sqrt{a})/F$.

\begin{example}[$k = n = 1$, nonsplit case] \label{nonsplit_example}
let $p$ be a prime integer. Suppose $L/F$ is a degree $p$ field
extension, and $A$ is a degree $p$ central simple $L$-algebra such that
$\mo{cor}_{L/F} A$ has index dividing $p$. Then there exists an \'etale
extension $E/F$ of degree $p^{p-1}m$ for some $m$ relatively prime to to
$p$ such that $A \otimes_F E$ is split. 
\end{example}

\begin{example}[$k = n = 1$, split case]\label{split_example}
Suppose $A_1, \ldots, A_p$ are algebras over a field $F$ of prime degree
$p$, such that $\mo{ind} A_1 \otimes \cdots \otimes A_p$ divides $p$.
Then there exists an \'etale extension $E/F$ of degree $p^{p-1}m$ for
some $m$ relatively prime to to $p$ such that $A \otimes_F E$ is split.
In particular, for $p = 3$ we find that $A$ has a degree $90 = 9 \cdot
10$ splitting field.
\end{example}

In the split case $n = k = 1$ above, this result can be seen to be sharp
in the sense of the following propositions:

\begin{prop*}[\ref{prop_1}]
There exists a field $F$ and central simple $F$-algebras $A_1, \ldots,
A_p$ such that $\mo{ind}(A_1 \otimes \cdots \otimes A_p) = p^2$, and every
field extension $E/F$ which splits each algebras $A_i$ has $p^p |
[E:F]$.
\end{prop*}

\begin{prop*}[\ref{prop_2}]
Let $p$ be a prime number, and choose positive integers $d, n$ with $d <
n < p$. Then there exists a field $F$ and central simple $F$-algebras
$A_1, \ldots, A_n$ of degree $p$ such that $\mo{ind}(A_1 \otimes \cdots
\otimes A_n) = p^d$ and every field extension $E/F$ which splits each
algebra $A_i$ has $p^n | [E:F]$.
\end{prop*}

In \cite{Kar:C2ind}, Karpenko shows that the generic division algebra of
index $p^n$ and period $p$ is indecomposable for any odd prime $p$. 
Since an algebra is decomposable if and only if it is a corestriction
with respect to a split \'etale extension (see lemma \ref{cor_product}),
it makes sense to generalize this result to try to show the generic
division algebra is not a corestriction for any \'etale extension:
\begin{thm*}[\ref{indecomp}]
Let $A$ be generic division algebra of degree $p^n$ and period $p$, and
let $F$ be the center of $A$. If $n < p^2$ then $A$ cannot be written as
$cor_{L/F} B$ for any \'etale extension $L/F$ and any Azumaya algebra
$B$ over $L$.
\end{thm*}

For $n \geq p^2$, for example if $A$ is a division algebra of index
$p^{p^2}$, the obstruction used in the proof of theorem \ref{indecomp}
to show $A$ is not a corestriction vanishes. This raises the question of
whether or not such an algebra really may be written as a corestriction.

\section{Preliminaries} \label{prelim}

To begin it will be necessary to develop some machinery for
understanding the corestriction of algebras and its relation to the
transfer of a scheme.

\subsection{Galois twists}

Let $F$ be a field, $L/F$ a Galois extension of separable algebras with
group $G$ and let $V$ be an $L$-module. For $\sigma$ in $G$, we define a
new $L$-module $\up\sigma V$ as follows. As a set, we define
\[\up\sigma V = \{\up\sigma(v) | v \in V\},\]
and we endow it with the operations $\up\sigma(v) + \up{\sigma}(w) =
\up\sigma(v+w)$ and for $x \in L$, we set $x \, \up\sigma(v) =
\up\sigma((\sigma^{-1} x) v)$. We let $\phi_\tau : \up\sigma V \to
\up{\tau\sigma} V$ denote the natural map $\phi_\tau(\up\sigma(v)) =
\up{\tau\sigma}(v)$. Note that this map is $\tau$-linear in the sense
that $\phi_\tau(x \, \up\sigma(v)) = \tau(x) \phi_\tau(\up\sigma(v))$.
One may also check that these maps satisfy $\phi_\sigma \phi_\tau =
\phi_{\sigma \tau} : \up\gamma V \to \up{\sigma\tau\gamma} V$.
By composing with these maps it is easy to check that there is a natural
isomorphism of bifunctors $Hom_L(\, \up\sigma V, W) = Hom_{L,\sigma}(V, W)$
giving an equivalence between $L$-linear maps from $\up\sigma V$ to $W$
and $\sigma$-linear maps from $V$ to $W$.

Regarding $V \mapsto \up\sigma V$ as a functor from the category of
$L$-modules to itself, we note that it is additive and monoidal. That
is, there are natural isomorphisms $\up\sigma(V \oplus W) = \up\sigma V
\oplus \up\sigma W$ and $\up\sigma(V \otimes_L W) = \up\sigma V
\otimes_L \up\sigma W$.

This definition may easily be extended to an additive and monoidal
functor from the category of $L$-algebras to itself. Suppose $A$ is an
$L$-algebra. Then we define the $L$ algebra $\up\sigma A$ to be the
algebra with underlying vector space as defined above, and with the
multiplication rule $\up\sigma(a)\up\sigma(b) = \up\sigma(a b)$. Note
that this amounts to the same thing as taking the same underlying ring
and taking the new $L$-structure map $L \to A$ to be the original one
composed with the automorphism $\sigma^{-1}$. As before, the maps
$\phi_\sigma$ make sense with the same definition and we again have a
natural isomorphism $Hom_L(\up\sigma A, B) = Hom_{L, \sigma}(A, B)$ of
algebra hom sets.

We may similarly extend this definition to schemes defined over $L$ by
patching over affine sets. For an $L$ scheme $X$, we denote the
resulting variety by $\up\sigma X$. By the previous paragraph this
amounts to the same thing as taking the same $F$-scheme and composing
the structure morphism with the map $\sigma^{-1}: \mo{Spec} L \to
\mo{Spec} L$. We also obtain a natural isomorphism $Hom_L(Y, \up\sigma
X) = Hom_{L, \sigma}(Y, X)$ as before.

\subsection{Coset twists}

We now extend these constructions to the case where $L/F$ is a separable
extension which is not necessarily Galois. Let $E/L$ be a Galois closure
of $L/F$ so that $E/F$ is Galois with group $G$ and let $H$ be the
subgroup fixing $L$. For an $L$ module $V$ and $\sigma \in G$, we define
the coset twist $\up{\sigma H}V$ (which will be an $E$-module) as
follows. As a set we define
\[\up{\sigma H} V = \frac{\{\, \up\tau(v) | \tau \in \sigma H, v \in V
\otimes_L E\}}{\sim}\]
Where the equivalence relation $\sim$ is defined by letting $\up\tau(v)
= \up\gamma(w)$ if and only if $\gamma^{-1}\tau(v) = w$. Note that
$\gamma^{-1}\tau \in H$ and since $v \in V \otimes_L E$, elements of $H$
act naturally via the second factor in the tensor product. We define the
$E$-module operations by setting $\up\tau(v) \up\gamma(w) =
\up\gamma((\gamma^{-1}\tau (v)) w)$ and $x \, \up\tau(v) =
\up\tau((\tau^{-1} x) v)$.  Note that the map $\up\sigma (V \otimes_L E)
\to \up{\sigma H} V$ defined by $\up\sigma(v) \mapsto \up\sigma(v)$ is
an isomorphism.

As before, we have natural morphisms $\phi_\sigma : \up{\tau H} V \to
\up{\sigma\tau H} V$ via $\phi_\sigma(\up\gamma(v)) =
\up{\sigma\gamma}(v)$. Once again, we may check $\phi_\sigma \phi_\tau =
\phi_{\sigma \tau}$.

This definition may be extended to algebras and varieties, and we make
free use of this fact.

For the ease of exposition for the proof of the following lemma we make
the following definition. Given a a scheme $X$ and vector bundles $W, V$
over $X$, we call an embedding $W \hra V$ admissible if $V/W$ is also a
vector bundle (and not merely a coherent sheaf). Equivalently, this says
that $W$ is locally a direct summand of $V$. In the case $W \subset V$,
we say that $W$ is admissible if the inclusion is admissible.

\begin{lem} \label{twist_SBV}
Suppose $A$ is an Azumaya algebra over $L$ and $m < \mo{deg}
A$. Then there is a natural isomorphism $\up{\sigma H} X_m(A) =
X_m(\up{\sigma H} A)$
\end{lem}
\begin{proof}
We will exhibit this by giving a natural isomorphism between the
functors which these schemes represent. For a $E$-scheme $Y$, recall that
$Hom_L(Y, X_m(A))$ may be thought of as the set of admissible subvector
bundles $I \subset Y \times_L A$ (where here $A$ is thought of as a vector
bundle over $\mo{Spec} L$) such that the sheaf corresponding to $I$ is a
sheaf of right ideals of rank $m \mo{deg} A$ in the sheaf of algebras
corresponding to $Y \times A$.

We have natural isomorphisms 
\begin{align*}
Hom_E(Y, \, \up{\sigma H} X_m(A)) = Hom_E(Y, \, \up\sigma (X_m(A)
\times_L E)) = Hom_{E, \sigma}(Y, X_m(A \otimes_L E)) \\
= Hom_E(\up{\sigma^{-1}} Y, X_m(A \otimes_L E))
\end{align*}
This last set may be identified with bundles of ideals $I \subset
\up\sigma^{-1} Y \times_E (A \otimes_L E)$, and these naturally
correspond (via application of the functor $\up\sigma(\,)$) to bundles
of ideals $I' \subset Y \times_E \up\sigma(A \otimes_L E)$ via $I
\mapsto I' = \up\sigma I$. We therefore have
\begin{align*}
Hom_E(\up{\sigma^{-1}} Y, X_m(A \otimes_L E)) = Hom_E(Y, X_m(\up\sigma(A
\otimes E))) = Hom_E(Y, X_m(\up{\sigma H} A))
\end{align*}
as desired.
\end{proof}

\subsection{Corestriction and transfer}

Suppose we have a separable extension of commutative rings $L/F$ and a
Galois closure $E/L$ with $G = Gal(E/F)$, $H$ the subgroup fixing $L$.
Given an $L$-algebra $A$, we define 
\[A^{E/L/F} = \underset{\sigma H}{\otimes} \, \up{\sigma H} A\]
where the tensor product (over the algebra $E$) is taken over all cosets
of $H$ in $G$. The group $G$ acts naturally on this algebra by
defining the action on simple tensors to be
\[\tau (\underset{\sigma H \in G/H}\otimes a_{\sigma H}) =
\underset{\sigma H \in G/H}\otimes \phi_{\tau}^{-1}(a_{\tau \sigma H})\] 
This action is semilinear in the sense that $\tau(x b) = \tau(x)
\tau(b)$. By the theory of Galois descent, the algebra $A^{E/L/F}$
together with the $G$ action give the descent data for an $F$ algebra
which we call the corestriction. Explicitly we may define the
corestriction of $A$ to be the fixed algebra $\mo{cor}_{L/F} A =
(A^{E/L/F})^G$.

\begin{lem} \label{cor_product}
Suppose $L = \prod_{i = 1}^k L_i$, and $A$ is an $L$ algebra. Then
writing $A_i = L_i A$, we have $\mo{cor}_{L/F} A = \otimes
\mo{cor}_{L_i/F} A_i$.
\end{lem}
\begin{proof}
We leave to the reader verification of the fact that for any \'etale
extensions $F \subset K \subset E$, that $\mo{cor}_{E/F} =
\mo{cor}_{K/F} \circ \mo{cor}_{E/K}$. Let $K$ be the split \'etale
extension $\prod_{i = 1}^k F$. It follows easily that $\mo{cor}_{E/K} =
\prod \mo{cor}_{L_i/F} A_i$. Therefore we only need to verify the
statement in the case that $L$ is itself a split algebra.

Now, assuming that $L_i = F$ for each $i$, we choose any group $G$ of
order $k$, and let it act transitively by permutations of the
idempotents of $L$. We may then regard $L/F$ as a $G$-Galois extension,
and so $\mo{cor}_{L/F} A = (A^{L/F})^G$. Thinking of $A = \prod A_i$, we
may write $A^{L/F}$ as a product $\prod_{\sigma \in G}
(A_1 \otimes \cdots \otimes A_k)$ where the elements of $G$ simply
permute the terms of the product. An element if fixed by the $G$ action
if it is in the image of the diagonal embedding $A_1 \otimes \cdots
\otimes \cdots A_k \hra A^{L/F}$, completing the proof.
\end{proof}

Similarly, if $E, L$ and $G$ are as above and we have an $L$-scheme $X$,
we define 
\[X^{E/L/F} = \underset{\sigma H}{\times} \, \up{\sigma H} X\]
with the fiber product taken with respect to $\mo{Spec} E$. As before we
have a natural Galois action by the group $G$, and so by Galois descent
it corresponds to an $F$-scheme which we denote $\mo{tr}_{L/F} X$.

\subsection{A bound on the degree of a splitting field}

Let $L$ be a finite commutative separable extension of $F$ and let $A$
be an Azumaya algebra over $L$. We do not assume that the algebra
$A$ has constant rank over $L$. We define $X(A)$, the Severi-Brauer
variety of $A$, to be the scheme parametrizing right ideals $I$ of $A$
such that for $p \in \mo{Spec} L$, $I_p$ is a right ideal of $A_p$ of
rank $\mo{deg} A_p$ (see \cite{Gro:GB1}).

\begin{prop} \label{a_priori_split}
Let $A$, $L$, $F$ be as above. Then there exists a commutative separable
extension $E/F$ with $[E:F] = \prod_p (\mo{deg} A_p)^{[F(p):F]}$ (where $p$
ranges through the points of $\mo{Spec} L$ and $F(p)$ is the residue
field of $p$) such that $A \otimes_F E$ is a split $L \otimes_F E$
algebra.
\end{prop}
\begin{proof}
Let $K \subset A$ be a maximal commutative separable subalgebra of $A$.
Since $A \otimes_L K$ is a split algebra over $K$, we obtain a morphism
of $L$-schemes $\mo{Spec} K \to \mo{Spec} X(A)$, and therefore a map
\[\mo{Spec} (\mo{cor}_{L/F} K) = \mo{tr}_{L/F} (\mo{Spec} K) \to
\mo{tr}_{L/F} X(A)\]

By the adjointness property of the transfer (see \cite{Serre:TG}), we
have \[Mor_F(\mo{Spec} (\mo{cor}_{L/F} K), \mo{tr}_{L/F} X(A)) =
Mor_L(\mo{Spec} (L \otimes_F \mo{cor}_{L/F} K), X(A))\]
and in particular, $X(A)$ has an $L \otimes_F \mo{cor}_{L/F} K$-point.
Setting $E = \mo{cor}_{L/F} K$, we now check that $E$ has the stated
dimension and note that $A \otimes_F E = A \otimes_L (L \otimes_F E)$ is
split.
\end{proof}

\subsection{Twisted Segre embeddings}

\begin{lem}
There is a natural closed embedding
\[\phi_{L/F}^A : \mo{tr}_{L/F} X(A) \to X(\mo{cor}_{L/F} A)\]
\end{lem}
\begin{proof}
Let $E/L$ be a Galois closure of $L/F$ with Galois group $G$ acting on
$E/F$ and with $H$ the subgroup fixing $L$.
We define this morphism by descent by constructing a morphism
\[\phi : X(A)^{E/L/F} \to X(\mo{cor}_{L/F} A) \times_F E =
X((\mo{cor}_{L/F} A) \otimes_F E) = X(A^{E/L/F})\]
Note that by lemma \ref{twist_SBV} we may write 
\[X(A)^{E/L/F} = \underset{\sigma H}\prod \, \up{\sigma H} X(A) =
\underset{\sigma H}\prod \, X(\up{\sigma H} A)\]
We define the map $\phi$ by sending a tuple of ideals indexed by the
cosets $G/H$, $(I_{\sigma H}) \in X(A)^{E/L/F} = \underset{\sigma
H}\prod \, X(\up{\sigma H} A)$ to the tensor product of the ideals
$\phi(I_{\sigma H}) = \underset{\sigma H}\otimes I_{\sigma H}$, and note
that this commutes with the natural action of the Galois group.
Therefore we obtain by descent our desired morphism.
\end{proof}

\begin{rem}
In the case that $L/F$ is a split \'etale extension, we have $A = A_1
\times \cdots \times A_m$ and the map $\phi_{L/F}^A$ may be written as
the map $X(A_1) \times \cdots \times X(A_m) \to X(A_1 \otimes \cdots
\otimes A_m)$ given by $(I_1, \ldots, I_m) \mapsto I_1 \otimes \cdots
\otimes I_m$. This map was investigated by Karpenko in
\cite{Kar:C2ind} - in particular, it was shown to be a twisted form
of the Segre embedding $\mbb P^{n-1} \times \mbb P^{m-1} \to \mbb P^{n^m
- 1}$.
\end{rem}

Let $X_{L/F}(A)$ denote the image of $\phi_{L/F}^A$. We define
$\mo{deg} X_{L/F}(A)$ to be the degree of the subvariety 
\[X_{L/F}(A)_{\ov F} \subset \mbb P^N\]
where $\ov{F}$ is an algebraic closure of $F$. For an algebra $A$ as
above, note that $A \otimes_L \ov{F}$ is a product of (split) algebras
$A_i$. 
\begin{defn} \label{deg_def}
Let $A$ be an Azumaya algebra over $L$, where $L/F$ is \'etale of
dimension $m$. Writing $A \otimes_F \ov{F} = A_1 \times \cdots \times
A_m$ for central simple algebras $A_i$, we define $\mo{deg} A$ to be the
(unordered) list of degrees $(\mo{deg} A_1, \mo{deg} A_2, \ldots,
\mo{deg} A_m)$.
\end{defn}

\begin{lem} \label{degree_lemma}
Suppose $L/F$ has degree $m$ and $A/L$ has degree $(d_1, \ldots, d_m)$.
Then
$\mo{deg} X_{L/F}(A)$ is the multinomial coefficient
\[\binom{d_1 + \cdots + d_m - m}{d_1 - 1, \ldots, d_m - 1} = \frac{(d_1
+ \cdots + d_m - m)!}{(d_1 - 1)! \cdots (d_m - 1)!}\]
\end{lem}
\begin{proof}
Without loss of generality, we may assume that $F$ is algebraically
closed. In this case, we are really considering the embedding of $Y =
\mbb P^{d-1 - 1} \times \cdots \times \mbb P^{d_m - 1}$ into $\mbb P^{d_1
\cdots d_m -1}$ via the Segre embedding, which we will denote by $\phi$.
Let $\ell_i$ be the pullback of the divisor $\mc O_{\mbb P^{d_i -
1}}(1)$ via the natural projection $Y \to \mbb P^{d_i - 1}$. Recall that
the Chow ring of $Y$ may be written $\mo{CH}(Y) = \mbb Z[\ell_1, \ldots,
\ell_m]/(\ell_i^{d_i})$, and that the Segre embedding is given by the
divisor $D = \phi^* \mc O_{\mbb P^{d_1 d_2 \cdots d_m - 1}}(1) = \sum_i
\ell_i$. If we set $d = d_1 + d_2 + \cdots + d_m - m = \mo{dim} Y$, then
the degree of the map is therefore given by the degree of $D^d$, the top
self intersection of the divisor $D$. By the presentation of the Chow
ring of $Y$ given above, it follows that the only term which is nonzero
in the multinomial expansion of $D^d = (\sum_i \ell_i)^d$ is the term
\[\binom{d}{d_1 - 1, \ldots, d_m - 1} \prod \ell_i^{d_i - 1}\]
and the fact that $\prod \ell_i^{d_i - 1}$ may be interpreted as the
class of a closed point in $Y$ immediately implies the result.
\end{proof}

\section{Is the generic algebra a corestriction?}

Let $A$ be a central simple algebra over $F$ and suppose that $A =
\mo{cor}_{L/F} B$ for some \'etale extension $L/F$ of degree $m$ and
Azumaya $L$ algebra of constant degree $d$. In particular, this implies
$\mo{deg} A = d^m$. Since the corestriction map on the level of
cohomology $\mo{cor} : \mo{Br}(L) \to \mo{Br}(F)$ is a homomorphism, it
follows that $\mo{per}(A) | \mo{per}(B)$ and so $\mo{per}(A) | d$. It
therefore makes sense to ask when the converse holds - namely, if $A$ is
an algebra of degree $d^m$ and period $d$, when is it a corestriction of
an \'etale extension of degree $m$?

A priori, this question is a bit more general than the one of
indecomposability, since one knows by lemma \ref{cor_product}, that if
an algebra is decomposable, it must also be a corestriction (with
respect to a split \'etale extension).

It turns out that in the case $d$ is an odd prime, the arguments of
Karpenko from \cite{Kar:C2ind} generalize nicely to handle
corestrictions as well as decomposability. The relevant result which we
quote is a special case a result of Karpenko's:

\begin{lem}\label{karpenko_lemma}
Let $p$ be a prime number, $n$ a positive integer. Let $D(p^n, p)$ be a
generic division algebra of degree $p^n$ and period $p$ and let $X$ be
its Severi-Brauer variety. Then for any cycle $Z \in CH^k(X)$, the
$p$-adic valuation of the degree of $Z$ is greater than or equal to the
minimum of the following set of numbers:
\[\{i + n -v_p(k - i) | i = 0, \ldots, k-1\} \cup \{k\}\]
where $v_p$ denotes the $p$-adic valuation. Furthermore, this remains
true even after a prime to $p$ extension.
\end{lem}
\begin{proof}
See \cite{Kar:C2ind}, proposition 1.3, and the proof of theorem 3.1.
\end{proof}

\begin{thm} \label{indecomp}
Let $A$ be generic division algebra of degree $p^n$ and period $p$, and
let $F$ be the center of $A$. If $n < p^2$ then $A$ cannot be written as
$cor_{L/F} B$ for any \'etale extension $L/F$ and any Azumaya algebra
$B$ over $L$.
\end{thm}
\begin{proof}
Suppose we have an algebra $B$ as above, and let $X = X(A)$ be the
Severi-Brauer variety of $A$. We claim that there exists some
cycle in $X(A)$ which contradicts lemma \ref{karpenko_lemma}, therefore
giving a contradiction. Since this would not be changed by prime to $p$
extensions, we may assume that $F$ has no separable field extensions whose
degree is not a power of $p$. Since it follows directly from
\cite{Kar:C2ind}, theorem 3.1 that $A$ is indecomposable, lemma
\ref{cor_product} implies that the extension $L$ must in fact be a field
and not just an \'etale algebra.  Consequently, the algebra $B$ has
constant rank $p^r$, and $[L:K] = p^s$ for some $r, s$. We therefore
also have $p^n = \mo{deg} A = (p^r)^{p^s} = p^{rp^s}$. By assumption, $n
< p^2$ implies that $s = 1$.

By lemma \ref{degree_lemma} and lemmas \ref{val_power}, \ref{val_misc},
the variety $\mo{tr}_{L/F} X(B)$ embeds as a subvariety in $X$ with
degree satisfying $v_p(\mo{deg} X) = rp - r$ and of codimension $p^{rp}
- p^r - p - 1$. To complete our argument by contradiction, we must now
show that the inequality implied by lemma \ref{karpenko_lemma} fails.
That is, we must show:
\begin{align*}
v_p(\mo{deg} X) = rp - r &< i + rp -v_p(p^{rp} - p^r - p - 1 - i), \\
rp - r &< p^{rp} - p^r - p - 1
\end{align*}
for $i$ between $0$ and $p^{rp} - p^r - p - 2$. For the second
inequality, by grouping terms and rewriting, we see that it is
equivalent to 
\[p^r(p^{r(p-1) - 1} > (p-1)(r+1) + 2\]
and $r(p-1) > 1$ (since $p$ is an odd prime) implies $p^{r(p-1)} - 1 >
p-1$, and so it is enough to show
\[p^r \geq r + 1 + \frac{2}{p-1}\]
for which it suffices to show $p^r \geq r+2$. But this is easy to check
(note that it holds for $r = 1$ and induction on $r$ implies quickly
that it holds for all $r$).

For the first inequality, we must show:
\[v_p(p^{rp} - p^r - p - 1 - i < r + i\]
If $i < p^r + p + 1$, then it follows that $v_p(p^{rp} - p^r - p - 1 -
i) < r$, and the inequality follows. Otherwise, $i \geq p^r + p + 1$ and
since $v_p(p^{rp} - p^r - p - 1 - i)$ is always less than $rp$, we must
only show $rp < r + p^r + p + 1$. We start by showing that $p^r \geq rp$
for all $r \geq 1$ by induction on $r$. For the induction step, we have:
\begin{multline*}
p^{r+1} = p^r + (p^{r+1} - p^r) = p^r + p^r(p - 1) \geq rp + rp(p-1)
> rp + p = (r+1)p
\end{multline*}
Therefore the inequality may be proved by showing $p + r + 1 > 0$, and
we are done.
\end{proof}

\section{Splitting fields of separable algebras}

\begin{lem} \label{splitting_lem}
Suppose $A$ is Azumaya over $L$ with $L/F$ is \'etale, and let $E/F$ be
another \'etale extension. Then $A \otimes_F E$ is a split algebra over
$L \otimes_F E$ if and only if the variety $\mo{tr}_{L/F} X(A)$ has an
$E$ point.
\end{lem}
\begin{proof}
By the adjointness property of the transfer (see \cite{Serre:TG}), we
have 
\[Hom_F(\mo{Spec} E, \mo{tr}_{L/F} X(A)) = Hom_L(\mo{Spec} (E \otimes_F
L), X(A))\]
but this in turn implies that 
\[X(A) \times_{\mo{Spec} L} \mo{Spec} (L \otimes_F E) = X(A \otimes_F
E)\] has an $L \otimes_F E$ point, which implies $A \otimes_F E$ is
split.
\end{proof}

\begin{thm} \label{general_thm}
Let $L/F$ an \'etale extension of dimension $m$ and and $A$ an Azumaya
algebra over $L$ of degree $(d_1, \ldots, d_m)$ (see definition
\ref{deg_def}). Let $I$ be the index of $\mo{cor}_{L/F} A$ and let $P$
be its period. Write $\sum d_i - m = qI + r$ for positive integers $q,
r$ with $r < I$. Then there exists an \'etale extension $E/F$ of degree
\[\frac{(d_1 + \cdots + d_m - m)!}{(d_1 - 1)! \cdots (d_m - 1)!} \,
P^r\]
such that $A \otimes_F E$ is split as an $L \otimes_F E$ algebra.
\end{thm}
\begin{proof}
Let $X = \mo{tr}_{L/F} X(A)$ and let $Y = X(\mo{cor}_{L/F} A)$. Since
$\mo{cor}_{L/F} A$ has index $I$, there exist cycles $Z^{qI} \subset Y$
such that $Z^{qI}_{\ov F}$ is isomorphic to a linear projective subspace
of $\bP^{d_1 \cdots d_m - 1} = Y_{\ov F}$ of codimension $qI$ (see
\cite{Ar:BS}). Also by \cite{Ar:BS}, there exists a
divisor $D \subset Y$ such that $\ov D = D_{\ov F}$ is in the class $P
[H]$ where $H$ is a hyperplane in $\mbb P^{d_1 \cdots d_m - 1}$. In
particular, intersecting a general subspace of the form $Z^{qI}$ with
one of the form $D^r$ will intersect $X$ in a subscheme $C \cong
\mo{Spec}(E)$, where $E/F$ is an \'etale extension of degree $(\mo{deg}
X) P^r$.

In particular, this means $X$ has an $E$ point our conclusion follows
from lemma \ref{splitting_lem}.
\end{proof}

\begin{thm} \label{main_thm}
Let $p$ be a prime number, $L/F$ an \'etale extension of degree $p^k$
and $A$ an Azumaya algebra over $L$ of constant degree $p^n$ such that
$\mo{cor}_{L/F} A$ has index dividing $p^k$. Then there exists an
\'etale extension $E/F$ of degree $p^{n(p^k - 1)}m$ where $m$ is
relatively prime to $p$ such that $A \otimes_F E$ is split as an $L
\otimes_F E$ algebra.
\end{thm}

\begin{proof}[proof of theorem \ref{main_thm}]
Let $X = \mo{tr}_{L/F} X(A)$ and let $Y = X(\mo{cor}_{L/F} A)$. Since
$\mo{cor}_{L/F} A$ has index $p^k$, there exist cycles $Z^{rp^k} \subset
Y$ such that $Z^{rp^k}_{\ov F}$ is isomorphic to a linear projective
subspace of $\bP^{(p^n)^{p^k} - 1}$ of codimension $rp^k$ for any $r$
(see \cite{Ar:BS}). In particular, a general subspace of the form
$Z^{p^{k+n} - p^k}$ will intersect $X$ in a subscheme $C \cong
\mo{Spec}(E)$, where $E/F$ is an \'etale extension of degree $\mo{deg}
X$.

By lemma \ref{splitting_lem} We therefore need only compute the
$p$-adic valuation of $\mo{deg} X$ to complete the proof. Using lemmas
\ref{val_power} and \ref{val_misc}, we have:
\[v_p((p^{k+n} - p^k)!) = v_p(p^{k+n}!) - v_p(p^k!) - n = \frac{p^{k+n}
- p^k}{p-1} - n\]
\[v_p((p^n-1)!) = v_p(p^n!) - v_p(p^n) = \frac{p^n - 1}{p-1} - n\]
and so using lemma \ref{degree_lemma}, we have:
\[v_p(\mo{deg} X) = v_p \binom{p^{k+n} - p^k}{p^n-1, \ldots, p^n-1}
=\frac{p^{k+n} - p^k}{p-1} - n - p^k(\frac{p^n - 1}{p-1} - n) 
=n(p^k - 1) \]
as desired.
\end{proof}

In the case that $L/F$ is a split \'etale extension, this gives the
following corollary:
\begin{cor} \label{main_cor}
Let $p$ be a prime number, and let $A_1, \ldots, A_{p^k}$ be central
simple algebras of degree $p^n$ over a field $F$ such that $A_1 \otimes
\cdots \otimes A_{p^k}$ has index dividing $p^k$. Then there exists an
\'etale extension $E/F$ of degree $p^{n(p^k-1)}m$ where $m$ is
relatively prime to $p$ such that $E$ is a splits each of the algebras
$A_1, \ldots, A_{p^k}$.
\end{cor}

\begin{cor} \label{quat_case}
Suppose $Q_1, Q_2$ are quaternion algebras over $F$ and $Q_1 \otimes
Q_2$ has index $2$. Then there is a common quadratic splitting extension
for $Q_1$ and $Q_2$.
\end{cor}
\begin{proof}
This follows immediately from theorem \ref{main_thm} and that in this
case $m = 1$ by remark \ref{what_is_m}.
\end{proof}

As a corollary, we present a proof of a result which is known to the
experts, but the present context provides a convenient method of proof:

\begin{prop}
Suppose $Q$ is a quaternion algebra, and $A$ is a division algebra of
degree $2m$ such that $A \otimes Q$ is not division. Then there is a
maximal subfield of $A$ which also splits $Q$.
\end{prop}
\begin{proof}
Let $L = F \times F$ be split quadratic \'etale, and $B = Q \times A$ as
an $L$ algebra. Note that $\mo{cor}_{L/F} B = Q \otimes A$ by lemma
\ref{cor_product}.  Let $X = X(B)$, $Y = X(Q \otimes A)$. By section
\ref{prelim}, we know that $X$ has dimension $2m$ and degree
$\binom{2m}{1, 2m-1} = 2m$.  Since $\mo{ind}(Q \otimes A) | 2m$, we have
a $Z \subset Y$ a form of a linear subspace of codimension $2m$. By
intersecting $Z$ with $X$, we obtain a $0$-dimensional subscheme of
$2m$. Using lemma \ref{splitting_lem}, we therefore obtain a splitting
field of degree $2m$.
\end{proof}

\section{Counterexamples}

The following lemma will be essential for the construction of
counterexamples. In its statement we will use the following notational
convention: if $A$ is a central simple $F$ algebra of degree $d$, and
$i$ is any integer, we will let $A^i$ denote the algebra of degree $d$
which is Brauer equivalent to $A^{\otimes i}$. This algebra is unique up
to isomorphism.
\begin{lem} \label{example_lem}
Let $d, n$ be positive integers. Then there exists a field $F$ and
central simple $F$ algebras $A_1, \ldots, A_n$ of degree $d$ such that
for any $n$-tuple $i_1, \ldots, i_n$ with $i_k$ relatively prime to $d$
for each $k$ the algebra $A_1^{i_1} \otimes \cdots \otimes A_n^{i_n}$ is
a division algebra.
\end{lem}
\begin{proof}
See \cite{Kar:CSF}, proposition III.1.
\end{proof}

\begin{prop} \label{prop_1}
There exists a field $F$ and central simple $F$-algebras $A_1, \ldots,
A_p$ such that $\mo{ind}(A_1 \otimes \cdots \otimes A_p) = p^2$, and every
field extension $E/F$ which splits each algebras $A_i$ has $p^p |
[E:F]$.
\end{prop}
\begin{proof}
To begin, choose a field $L$ and central simple algebras $A_1, \ldots,
A_p$ of degree $p$ satisfying the properties described in lemma
\ref{example_lem}. Let $A = A_1 \otimes \cdots \otimes A_p$ and $A' =
A_1 \otimes A_2 \otimes A_3^2 \otimes A_4^3 \otimes \cdots \otimes
A_p^{p-1}$. Let $F$ be the function field of the generalized
Severi-Brauer variety $X_{p^2} A$ of $p^3 = p \cdot p^2$ dimensional
right ideals of $A$. 

The algebra $A \otimes F$ has index $p^2$, but we claim that the algebra
$A' \otimes F$ has index $p^p$. This would prove our claim since if
$E/F$ is any common splitting field of $A_1 \otimes F, \ldots, A_p
\otimes F$, it would also have to be a splitting field of the algebra
$A' \otimes F$ and hence have degree divisible by $p^p$.

We may compute $\mo{ind}(A' \otimes F)$ using the index
reduction formula of Blanchet (\cite{Blanchet}) as phrased in
\cite{MPW}:
\[\mo{ind}(A' \otimes F) = \underset{1 \leq i \leq }{\gcd} \left\{
\frac{p^2}{\gcd\{p^2, i\}} \mo{ind} A' \otimes A^i \right\}\]
and by our construction of the algebras $A$ and $A'$ we may compute:
\[\frac{p^2}{\gcd\{p^2, i\}} \mo{ind} A' \otimes A^i = \left\{
\begin{matrix}
p^2 \cdot p^{p-2} && \text{if $i = p-1$} \\
p^2 \cdot p^{p-1} && \text{if $i \neq p-1$ and $p \not| i$} \\
p \cdot p^p \text{ or } 1 \cdot p^p && \text{if $p | i$}
\end{matrix}
\right.\]
and in particular, $p^p = \mo{ind} A' \otimes F$.
\end{proof}

\begin{prop} \label{prop_2}
Let $p$ be a prime number, and choose positive integers $d, n$ with $d <
n < p$. Then there exists a field $F$ and central simple $F$-algebras
$A_1, \ldots, A_n$ of degree $p$ such that $\mo{ind}(A_1 \otimes \cdots
\otimes A_n) = p^d$ and every field extension $E/F$ which splits each
algebra $A_i$ has $p^n | [E:F]$.
\end{prop}
\begin{proof}
As in the previous lemma, choose a field $L$ and central simple algebras
$A_1, \ldots, A_n$ of degree $p$ satisfying the properties described in
lemma \ref{example_lem}. Let $A = A_1 \otimes \cdots \otimes A_n$ and
$A' = A_1 \otimes A_2^2 \otimes \cdots \otimes A_p^p$. Let $F$ be the
function field of $X_{p^d}(A)$.

The algebra $A \otimes F$ has index $p^d$, but we claim that the algebra
$A' \otimes F$ has index $p^n$. As in the previous lemma, this would
prove our claim. Since index may only decrease upon scalar extensions,
it suffices to show that the algebra $A' \otimes L$ has index a multiple
of $p^n$ for $L$ the function field of $X_p(A)$.

We may compute $\mo{ind}(A' \otimes F)$:
\[\mo{ind}(A' \otimes F) = \underset{1 \leq i \leq }{\gcd} \left\{
\frac{p}{\gcd\{p, i\}} \mo{ind} A' \otimes A^i \right\}\]
and by our construction of the algebras $A$ and $A'$ we may compute:
\[\frac{p}{\gcd\{p, i\}} \mo{ind} A' \otimes A^i = \left\{
\begin{matrix}
p \cdot p^{n-1} \text{ or } p \cdot p^n && \text{if $i \neq p$} \\
p^{n} && \text{if $i = p$}
\end{matrix}
\right.\]
and in particular, $p^n = \mo{ind} A' \otimes L$ as claimed.
\end{proof}

\section{Counting lemmas}

\begin{lem} \label{val_prod}
Suppose $p$ is a prime integer and $1 \leq k < p$. Let $v_p$ denote the
$p$-adic valuation. Then $v_p((kp^n)!) = k v_p(p^n!)$.
\end{lem}
\begin{proof}
We may write $(kp^n)! = \prod_{i = 0}^{k-1} \prod_{j = 1}^{p^n} (ip^n +
j)$. Noting that $v_p(ip^n + j) = v_p(j)$
and taking valuations of both sides gives:
\[v_p((kp^n)!) = \sum_{i = 0}^{k-1} \sum_{j=1}^{p^n} v_p(ip^p + j) = 
\sum_{i = 0}^{k-1} \sum_{j=1}^{p^n} v_p(j) = k v_p(p^n!)\]
\end{proof}

\begin{lem} \label{val_power}
Suppose $p$ is a prime integer. Then \[v_p(p^n!) = \frac{p^n - 1}{p -
1}.\]
\end{lem}
\begin{proof}
We may show this by induction, the case of $n = 1$ being left to the
reader. For the induction case, suppose that the result holds for $n$.
In this case, we write $p^{n+1}! = (p^{n+1} - p^n)!
\prod_{i=1}^{p^n}(p^{n+1} - p^n + i)$. Noting that 
\[v_p(p^{n+1} - p^n + i) = \left\{
\begin{matrix}
v_p(i) && \text{if $i \neq p^n$} \\
n+1 && \text{if $i = p^n$}
\end{matrix} \right.\]
and by lemma \ref{val_prod} that 
\[v_p((p^{n+1} - p^n)!) = v_p\big(((p-1)p^n)!\big) = (p-1)v_p(p^n!),\] 
we have
\begin{multline*}
v_p(p^{n+1}!) = v_p((p^{n+1} - p^n)!) + \sum_{i=1}^{p^n} v_p(p^{n+1} -
p^n + i) \\
= (p-1)v_p(p^n!) + n + 1 + \sum_{i=1}^{p^n - 1} v_p(i) = (p-1) v_p(p^n!)
+ v_p(p^n) + 1 + v_p((p^n - 1)!) \\
= (p-1)v_p(p^n!) + v_p(p^n!) + 1 
= p v_p(p^n!) + 1 = \frac{p^{n+1} - p}{p-1} + 1 = \frac{p^{n+1} - 1}{p -
1}
\end{multline*}
\end{proof}

\begin{lem} \label{val_misc}
Let $p$ be a prime integer. Then:
\[v_p\big((p^k(p^n - 1))!\big) = v_p(p^{k+n}!) - v_p(p^k!) - n\]
\end{lem}
\begin{proof}
We have:
\begin{multline*}
v_p(p^{n+k}!) = v_p\big((p^{n+k} - p^k)!\big) +
\left(\underset{i=1}{\overset{p^k-1}\sum} v_p(p^{n+k} - p^k + i)\right)
+ v_p(p^{n+k}) \\
= v_p\big((p^{n+k} - p^k)!\big) +
\left(\underset{i=1}{\overset{p^k-1}\sum} v_p(i)\right) + n+k 
= v_p\big((p^{n+k} - p^k)!\big) + v_p(p^k!) + n
\end{multline*}
\end{proof}

\bibliographystyle{alpha}
\bibliography{citations}

\end{document}